\documentclass[a4,12pt]{amsart}
\usepackage{amssymb,amstext,amsmath,amscd,amsthm,amsfonts,enumerate,graphicx,latexsym}
\usepackage[usenames]{color}
\usepackage[all]{xy}
\newtheorem{thm}{Theorem}[section]
\newtheorem{cor}[thm]{Corollary}
\newtheorem{lem}[thm]{Lemma}
\newtheorem{prop}[thm]{Proposition}
\theoremstyle{definition}

\newtheorem{rem}[thm]{Remark}

\newtheorem{ex}[thm]{Example}

\newtheorem*{claim*}{Claim}
\theoremstyle{remark}
\newtheorem*{ac}{Acknowlegments}
\numberwithin{equation}{thm}
\def\Hom{\operatorname{\mathsf{Hom}}}
\def\sHom{\operatorname{\mathsf{\underline{Hom}}}}
\def\syz{\mathsf{\Omega}}
\def\Ext{\operatorname{\mathsf{Ext}}}
\begin{document}
\allowdisplaybreaks
\setlength{\baselineskip}{14.6pt}
\title{Syzygies of Cohen-Macaulay modules and Grothendieck groups}
\date{}
\author{Toshinori Kobayashi}
\address{Graduate School of Mathematics, Nagoya University, Furocho, Chikusaku, Nagoya, Aichi 464-8602, Japan}
\email{m16021z@math.nagoya-u.ac.jp}
\thanks{2010 {\em Mathematics Subject Classification.} 13C14, 13D15, 13H10}
\thanks{{\em Key words and phrases.} Cohen--Macaulay ring,
Cohen-Macaulay module, Grothendieck group}
\begin{abstract}
We study the converse of a theorem of Butler and Auslander-Reiten. We show that a Cohen-Macaulay local ring with an isolated singularity has only finitely many isomorphism classes of indecomposable summands of syzygies of Cohen-Macaulay modules if the Auslander-Reiten sequences generate the relation of the Grothendieck group of finitely generated modules. This extends a recent result of Hiramatsu, which gives an affirmative answer in the Gorenstein case to a conjecture of Auslander.
\end{abstract}
\maketitle
%\tableofcontents
%%%%%%%%%%%%%%%%%%%%%%%%%%%%%%%%%%%%%%%%%%%%%%%%%%%%%%%%
\section{Introduction}

Throughout this note, let $(R,\mathfrak{m},k)$ be a Cohen-Macaulay local ring with an isolated singularity. We denote by $\mathsf{CM}(R)$ (resp. $\mathsf{mod}(R)$) the category of (maximal) Cohen-Macaulay $R$-modules (resp. finitely generated $R$-modules) with $R$-homomorphisms.

Let $\mathsf{G}(\mathsf{CM}(R))$ be the quotient of the free abelian group $\bigoplus \mathbb{Z}[X]$ generated by the isomorphism classes $[X]$ of modules $X$ in $\mathsf{CM}(R)$ by the subgroup generated by
\[
\{[X]+[Z]-[Y] \mid Y\cong X\oplus Z\}.
\]
Thus $\mathsf{G}(\mathsf{CM}(R))$ is isomorphic to the free abelian group generated by the isomorphism classes of indecomposable Cohen-Macaulay $R$-modules.

We denote by $\mathsf{Ex}(\mathsf{CM}(R))$ the subgroup of $\mathsf{G}(\mathsf{CM}(R))$ generated by
\[
\{[X]+[Z]-[Y] \mid \text{there exists an exact sequence }0 \to Z \to Y \to X \to 0\text{ in }\mathsf{CM}(R)\}.
\]
Then the quotient group $\mathsf{G}(\mathsf{CM}(R))/\mathsf{Ex}(\mathsf{CM}(R))$ is nothing but the Grothendieck group $\mathsf{K}_0(\mathsf{CM}(R))$ of $\mathsf{CM}(R)$ and therefore coincides with the Grothendieck group of $\mathsf{mod} (R)$.

We also denote by $\mathsf{AR}(\mathsf{CM}(R))$ the subgroup of $\mathsf{G}(\mathsf{CM}(R))$ generated by
\[
\left\{[X]+[Z]-[Y] \middle|
\begin{array}{l}
\text{there exists an Auslander-Reiten sequence }\\
0 \to Z \to Y \to X \to 0 \text{ in }\mathsf{CM}(R)
\end{array}
\right\}.
\]

Concerning the relationship between $\mathsf{Ex}(\mathsf{CM}(R))$ and $\mathsf{AR}(\mathsf{CM}(R))$, the following theorem holds; see \cite{B}, \cite[Proposition 2.2]{AR} and \cite[Theorem 13.7]{Y}.

\begin{thm}[Butler, Auslander-Reiten] \label{01}
If $R$ is of finite $\mathsf{CM}$ type, then $\mathsf{Ex}(\mathsf{CM}(R))=\mathsf{AR}(\mathsf{CM}(R))$.
\end{thm}

Here we say that $R$ is of {\it finite $\mathsf{CM}$ type} if there are only finitely many isomorphism classes of indecomposable Cohen-Macaulay $R$-modules.

Auslander conjectured that the converse of Theorem \ref{01} holds. Our main result is the following theorem, which yields a weaker version of the converse of Theorem \ref{01}.

\begin{thm}\label{1}
If $\mathsf{Ex}(\mathsf{CM}(R))\otimes_{\mathbb{Z}} \mathbb{Q}=\mathsf{AR}(\mathsf{CM}(R))\otimes_{\mathbb{Z}} \mathbb{Q}$, then there exist only finitely many isomorphism classes of indecomposable summands of (first) syzygies of Cohen-Macaulay $R$-modules.
\end{thm}

If $R$ is Gorenstein, then every Cohen-Macaulay $R$-module is a first syzygy of some Cohen-Macaulay $R$-module. Hence Theorem \ref{1} recovers the following result, which is proved by Hiramatsu \cite{H} and gives an affirmative answer to Auslander's conjecture in the case of Gorenstein local rings.

\begin{cor}[Hiramatsu]
Assume that $R$ is Gorenstein. If $\mathsf{Ex}(\mathsf{CM}(R))=\mathsf{AR}(\mathsf{CM}(R))$, then $R$ is of finite $\mathsf{CM}$ type.
\end{cor}

When $R$ is a two dimensional complete local ring and $k$ is algebraically closed, it is shown in \cite[Corollary 3.3]{DITV} that $R$ has a finite number of isomorphism classes of indecomposable summands of syzygies of Cohen-Macaulay modules if and only if $R$ is a rational singularity. This fact provides the following corollary.

\begin{cor}
Assume that $R$ is a complete local ring of dimension two with $k$ algebraically closed. If $\mathsf{Ex}(\mathsf{CM}(R))\otimes_{\mathbb{Z}} \mathbb{Q}=\mathsf{AR}(\mathsf{CM}(R))\otimes_{\mathbb{Z}} \mathbb{Q}$, then $R$ is a rational singularity.
\end{cor}

In the rest of this note, we give a proof of Theorem \ref{1}.

\section{Proof of our theorem}

As in the introduction, we always assume that $(R,\mathfrak{m},k)$ is a Cohen-Macaulay local ring with an isolated singularity. All $R$-modules are assumed to be finitely generated.

We denote by $\underline{\mathsf{mod}} (R)$ (resp. $\underline{\mathsf{CM}} (R)$) the stable category of $\mathsf{mod}(R)$ (resp. $\mathsf{CM}(R)$). These categories are defined in such a way that the objects are the same as those of $\mathsf{mod}(R)$ (resp. $\mathsf{CM}(R)$), and for objects $M, N$, the set of morphisms from $M$ to $N$ is $\sHom_R(M,N)$, defined to be the quotient of $\mathsf{Hom}_R(M,N)$ by the $R$-submodule consisting of homomorphisms factoring through free $R$-modules.

To give a proof of Theorem \ref{1}, we prepare several lemmas. The first one is given in \cite[Lemma 2.1]{H}.

\begin{lem} \label{2}
There exists a Cohen-Macaulay $R$-module $X$ such that for any non-free Cohen-Macaulay $R$-module $M$ one has $\sHom_R(M,X)\not=0$.
\end{lem}

We denote by $\syz M$ the first syzygy of $R$-module $M$, and by $\mathsf{Tr} M$ the (Auslander) transpose of $M$; see \cite[Definition (3.5)]{Y}. The modules $\syz M$ and $\mathsf{Tr}$ are uniquely determined by $M$ up to free summands.

We have the lemma below; see \cite[Proposition 2.7]{MT} for instance.

\begin{lem} \label{3}
Let $0 \to A \to B \to C \to 0$ be an exact sequence in $\mathsf{mod}(R)$ and $U$ an $R$-module. Then there exists a long exact sequence
\begin{align*}
\sHom_R(U,\syz A) \to \sHom_R(U,\syz B) \to \sHom_R(U, \syz C) \to \\
\sHom_R(U, A) \to \sHom_R(U,B) \to \sHom_R(U,C).
\end{align*}
\end{lem}

Here, we define $\syz^{-1} M$ to be the module $\mathsf{Tr}\syz\mathsf{Tr} M$ for an $R$-module $M$. The assignments $M \mapsto\syz M$, $M \mapsto \mathsf{Tr} M$ and $M \mapsto \syz^{-1} M$ define additive endofunctors of $\underline{\mathsf{mod}} (R)$. Moreover, the following lemma holds; see \cite[Corollary 3.3]{AR2}.

\begin{lem} \label{4}
Let $X,Y$ be $R$-modules. Then there is an isomorphism 
\[
\sHom_R(X,\syz Y) \to \sHom_R(\syz^{-1}X, Y),
\]
which is natural in $X,Y$ (i.e. one has an adjoint pair $(\syz^{-1}, \syz): \underline{\mathsf{mod}}(R) \to \underline{\mathsf{mod}}(R)$.)
\end{lem}

We denote by $\mathsf{\Omega CM}(R)$ the full subcategory of $\mathsf{mod}(R)$ consisting of first syzygies of Cohen-Macaulay $R$-modules. 

\begin{rem}
It is easy to see that $\mathsf{\Omega CM}(R)$ is closed under direct summands. In particular, the following are equivalent.
\begin{enumerate}[\rm(1)]
\item
There are only finitely many non-isomorphic indecomposable modules in $\mathsf{\Omega CM}(R)$.
\item
There are only finitely many non-isomorphic indecomposable summands of modules in $\mathsf{\Omega CM}(R)$.
\end{enumerate}
If one/both of these conditions is/are satisfied, we say that $R$ is {\it of finite $\mathsf{\Omega CM}$ type}.
\end{rem}

Now we give some properties of modules in $\mathsf{\Omega CM}(R)$.

\begin{lem} \label{04}
If $M \in \mathsf{\Omega CM}(R)$, then $\syz^{-1} M$ is in $\mathsf{CM}(R)$ and $M \cong \syz \syz^{-1} M$ up to free summands.
\end{lem}

\begin{proof}
Since $M$ is in $\mathsf{\Omega CM}(R)$, there is a Cohen-Macaulay module $N$ such that $M$ is the first syzygy of $N$. By the proof of \cite[Proposition 2.21]{AB}, there is an exact sequence $0 \to R^{\oplus a} \to \syz^{-1} M \oplus R^{\oplus b} \to N \to 0$ with some integers $a, b$. This implies that $\syz^{-1} M$ is a Cohen-Macaulay module. As $M$ is a syzygy, $M$ is isomorphic to $\syz\syz^{-1} M=\mathsf{\Omega Tr \Omega Tr} M$ up to free summands by \cite[Theorem 2.17]{AB}.
\end{proof}

Next we investigate the non-free part of a given module.

\begin{lem} \label{004}
Let $M$ be a finitely generated $R$-module, and $\widehat{R}$ the completion of $R$.
\begin{enumerate}[\rm(1)]
\item
$M$ has an $R$-free summand if and only if $M\otimes_R \widehat{R}$ has a $\widehat{R}$-free summand. 
\item
There is a unique decomposition $M \cong \underline{M}\oplus F$ of $M$ up to isomorphism with $F$ free such that $\underline{M}$ has no free summands. We call this module $\underline{M}$ the {\rm non-free part} of $M$.
\item
Let $A,B$ be finitely generated $R$-modules. If $A$ is a direct summand of $B$, then $\underline{A}$ is a direct summand of $\underline{B}$.
\end{enumerate}
\end{lem}

\begin{proof}
(1) The assertion follows from  \cite[Corollary 1.15 (i)]{LW}.

(2) We can take a maximal free summand $R^{\oplus a}$ of $M$ to have a decomposition $M \cong M'\oplus R^{\oplus a}$ where $M'$ has no free summands. Suppose that there is another decomposition $M \cong M''\oplus R^{\oplus b}$ where $M''$ has no free summands. Taking the completion, we have $M\otimes_R \widehat{R} \cong (M'\otimes_R \widehat{R}) \oplus \widehat{R}^{\oplus a} \cong (M'' \otimes_R \widehat{R}) \oplus \widehat{R}^{\oplus b}$. By (1), $M'\otimes_R \widehat{R}$ and $M''\otimes_R \widehat{R}$ have no free summands. Since the Krull-Schmidt property holds over $\widehat{R}$, we have $M'\otimes_R \widehat{R} \cong M'' \otimes_R \widehat{R}$ and $a=b$. Using \cite[Corollary 1.15 (ii)]{LW}, we have $M'\cong M''$.

(3) Suppose that $A$ is a direct summand of $B$. Then $\underline{A}$ is also a direct summand of $B$. Hence we have a decomposition $B \cong \underline{A} \oplus C$. It follows from (2) that the non-free part $\underline{B}$ of $B$ is isomorphic to the module $\underline{A}\oplus \underline{C}$. In particular, $\underline{B}$ has $\underline{A}$ as a direct summand. 
\end{proof}

Since there is an isomorphism $\sHom_R(M,N) \cong \mathsf{Tor}^R_1(\mathsf{Tr}(M),N)$ for finitely generated $R$-modules $M$, $N$ (see \cite[Lemma (3.9)]{Y}) and since we assume that $R$ is an isolated singularity, we can show that the length of the $R$-module $\sHom_R(M,N)$ is finite for any $M, N$ in $\mathsf{CM}(R)$. We denote by $[M,N]$ the integer $\mathsf{length}_R(\sHom_R(M,N))$. The following proposition plays a key role in the proof of our theorem. For the definition and basic properties of an Auslander-Reiten sequence, we refer the reader to \cite{LW, Y}.

\begin{prop} \label{5}
Let $0 \to A \to B \to C \to 0$ be an Auslander-Reiten sequence, and $U$ be a non-free Cohen-Macaulay $R$-module. Then the following hold.
\begin{enumerate}[\rm(1)]
\item
The induced sequence $\sHom_R(U,B) \to \sHom_R(U,C) \to 0$ is exact if and only if $C$ is not a direct summand of $U$.
\item
Suppose that $U$ is an indecomposable module in $\mathsf{\Omega CM}(R)$. 
\begin{enumerate}[\rm(a)]
\item
If $U$ is not isomorphic to the non-free part of $\syz C$, then the induced sequence $0 \to \sHom_R(U,A) \to \sHom_R(U,B)$ is exact.
\item
If $[U,C]+[U,A]-[U,B]\not=0$, then $U$ is isomorphic to either $C$ or the non-free part of $\syz C$.
\end{enumerate}
\end{enumerate}
\end{prop}

\begin{proof}
(1) Assume that $C$ is not a direct summand of $U$. Using the lifting property of an Auslander-Reiten sequence, every homomorphism from $U$ to $C$ factors through the map $B \to C$. This means that $\sHom_R(U,B) \to \sHom_R(U,C)$ is surjective. Conversely, suppose that $C$ is a direct summand of $U$. Then there is a split epimorphism $f:U \to C$. Let $g:C \to U$ be the right-inverse of $f$. If $\sHom_R(U,B) \to \sHom_R(U,C)$ is surjective, then so is the morphism $\Hom_R(U,B) \to \Hom_R(U,C)$ because of the commutative diagram
\[
\xymatrix{
\Hom_R(U,B) \ar@{->}[d] \ar@{->}[r] & \Hom_R(U,C) \ar@{->}[d] \ar@{->}[r] & \Ext_R^1(U,A) \ar@{=}[d] \\
\sHom_R(U,B) \ar@{->}[r] \ar[d] & \sHom_R(U,C) \ar[d] \ar@{->}[r] & \Ext_R^1(U,A)\\
0 & 0
}
\]
with exact rows and columns; see the proof of \cite[Proposition 2.7 (2)]{MT}. Hence there is a lift $h:U \to B$ of $f$. The morphism $hg:C \to B$ is the right-inverse of the morphism $B \to C$. This contradicts the definition of the Auslander-Reiten sequences.

(2a) Let $K$ be the kernel of the map $\sHom_R(U,A) \to \sHom_R(U,B)$. By Lemma \ref{3}, the sequence $\sHom_R(U,\syz B)\to \sHom_R(U,\syz C) \to K \to 0$ is exact. Then the equality $K=0$ is equivalent to the exactness of $\sHom_R(\syz^{-1}U, B)\to \sHom_R(\syz^{-1}U, C)\to 0$, by using Lemma \ref{4}. Lemma \ref{04} implies that $\syz^{-1} U$ is Cohen-Macaulay. By (1), $K$ is zero if and only if $C$ is not a direct summand of $\syz^{-1} U$.

Suppose that $C$ is a direct summand of $\syz^{-1} U$. By Lemma \ref{04} there is a free $R$-module $F$ such that $\syz C$ is a direct summand of $U\oplus F$. Lemma \ref{004} (3) implies that the non-free part of $\syz C$ is a direct summand of $U$, and is isomorphic to $U$ as $U$ is indecomposable.

(2b) Let $L$ be the cokernel of the map $\sHom_R(U,B) \to \sHom_R(U,C)$. Then we have an equality $[U,C]+[U,A]-[U,B]=\mathsf{length}_R(L)+\mathsf{length}_R(K)$. By (1), $\mathsf{length}_R(L)\not=0$ implies that $U$ is isomorphic to $C$. By (2a), $\mathsf{length}_R(K)\not=0$ implies that $U$ is isomorphic to the non-free part of $\syz C$.
\end{proof}

Now we can give a proof of our theorem.

\begin{proof}[Proof of Theorem \ref{1}]
Let $X$ be the module that satisfies the conditions in Lemma \ref{2}. Then there is an exact sequence with a free module $P$:
\[
0 \to \syz X \to P \to X \to 0.
\]
Since $\mathsf{Ex}(\mathsf{CM}(R))\otimes_{\mathbb{Z}} \mathbb{Q}=\mathsf{AR}(\mathsf{CM}(R))\otimes_{\mathbb{Z}} \mathbb{Q}$, there are a finite number of indecomposable Cohen-Macaulay $R$-modules $C_1,\dots,C_n$ and an equality in $\mathsf{G}(\mathsf{CM}(R))$:
\[
a([X]+[\syz X]-[P])=\sum_{i=1}^n b_i([A_i]+[C_i]-[B_i]),
\]
where $a$ is a positive integer, $b_i$ are integers and $[A_i]+[C_i]-[B_i]$ come from Auslander-Reiten sequneces $0 \to A_i \to B_i \to C_i \to 0$.
We have an equality in $\mathbb{Z}$:
\begin{equation}\label{7}
a[U,X \oplus \syz X]=\sum_{i=1}^nb_i([U,A_i]+[U,C_i]-[U,B_i])
\end{equation}
for each non-free indecomposable module $U$ in $\mathsf{\Omega CM}(R)$, since in general the equality $[U, \bigoplus_{s=1}^t L_s]=\sum_{s=1}^t [U,L_s]$ holds for any $R$-modules $L_1,\dots,L_t$. The left-hand side of \eqref{7} is nonzero by the choice of $X$ in Lemma \ref{2}, and hence so is the right-hand side. By Proposition \ref{5} (2b), this can occur only when $U$ is isomorphic to either $C_i$ or the non-free part of $\syz C_i$ for some $i$, and we conclude that the number of isomorphism classes of such modules $U$ is finite.
\end{proof}

\begin{rem}
The converse of Theorem \ref{01} has been proved by Auslander \cite{A} for artin algebras and by Auslander-Reiten \cite{AR2} for one dimensional complete local domains. We shall give examples of finite dimensional local algebras and one dimensional complete local domains which are of finite $\mathsf{\Omega CM}$ type but not of finite $\mathsf{CM}$ type. Thus finite $\mathsf{\Omega CM}$ type is not sufficient to hold the equality $\mathsf{EX}(\mathsf{CM}(R))=\mathsf{AR}(\mathsf{CM}(R))$, and the converse of Theorem \ref{1} is not true in general if we replace the condition $\mathsf{EX}(\mathsf{CM}(R))\otimes_{\mathbb{Z}}\mathbb{Q}=\mathsf{AR}(\mathsf{CM}(R))\otimes_{\mathbb{Z}}\mathbb{Q}$ with the condition $\mathsf{EX}(\mathsf{CM}(R))=\mathsf{AR}(\mathsf{CM}(R))$.
\end{rem}

\begin{ex}
Let $R=k[X,Y]/(X,Y)^2$ with $k$ a field. Then $R$ is an finite dimensional local $k$-algebra and not of finite $\mathsf{CM}$ type. Since the first syzygy $M$ of a non-free $R$-module is a submodule of a direct sum of copies of the maximal ideal $\mathfrak{m}$, the module $M$ is annihilated by $\mathrm{ann}(\mathfrak{m})=\mathfrak{m}$. So $M$ is a module over $R/\mathfrak{m}$. In particular, every non-free indecomposable module in $\mathsf{\Omega CM}(R)$ is isomorphic to $R/\mathfrak{m}$, and $R$ is of finite $\mathsf{\Omega CM}$ type.
\end{ex}

\begin{ex} \label{20}
Let $S=k[[t]]$ with $k$ a field, $n\geq 1$ be an integer and $R=k[[t^n,t^{n+1},\dots,t^{2n-1}]]$ be the subring of $S$. Then $R$ is a one dimensional complete local domain and the maximal ideal $\mathfrak{m}=t^nS$ of $R$ is isomorphic to $S$ as an $R$-module. Let $M$ be a non-free indecomposable module in $\mathsf{\Omega CM}(R)$. We  show that $M$ can be regard as an $S$-submodule of some free $S$-module. In fact, there exist a Cohen-Macaulay $R$-module $N$ and a short exact sequence $0 \to M \to R^{\oplus a} \to N \to 0$ coming from a minimal free resolution of $N$. By the minimality, we have a commutative diagram with exact rows
\[
\xymatrix{
0 \ar@{->}[r] & M \ar@{->}[r]  \ar@{=}[d] & \mathfrak{m}^{\oplus a} \ar@{->}[d] \ar@{->}[r] & L \ar@{->}[r] \ar@{->}[d] & 0\phantom{.}\\
0 \ar@{->}[r] & M \ar@{->}[r] & R^{\oplus a} \ar@{->}[r] & N \ar@{->}[r] & 0.
}
\]
By the snake lemma, $L$ is viewed as a submodule of $N$ and thus a Cohen-Macaulay $R$-module. Replacing $\mathfrak{m}$ with $S$ and multiplying $t^n$, we get a commutative diagram
\[
\xymatrix{
0 \ar@{->}[r] & M \ar@{->}[r]  \ar@{->}[d]^{t^n} & S^{\oplus a} \ar@{->}[d]^{t^n} \ar@{->}[r] & L \ar@{->}[r] \ar@{->}[d]^{t^n} & 0\\
0 \ar@{->}[r] & M \ar@{->}[r] \ar@{->}[d] & S^{\oplus a} \ar@{->}[r] \ar@{->}[d] & L \ar@{->}[r] & 0\\
& M/t^nM \ar@{->}[r] \ar@{->}[d] & S^{\oplus a}/t^nS^{\oplus a} \ar@{->}[d] & &\\
& 0 & 0 & &
}
\]
where the rows and columns are both exact. Applying the snake lemma again, we see that the morphism $M/t^nM \to S^{\oplus a}/t^nS^{\oplus a}$ in the diagram above is injective, as $L$ is a Cohen-Macaulay $R$-module. Since $t^{n+1}$ annihilates $S^{\oplus a}/t^nS^{\oplus a}$, it also annihilates $M/t^nM$. Hence $t^{n+1}M\subset t^nM$. Identifying $M$ as an $R$-submodule of $S^{\oplus a}$, we observe $tM\subset M$, which makes $M$ be an $S$-submodule of $S^{\oplus a}$. Since $S$ is a discrete valuation ring, the submodule $M$ of the free $S$-module $S^{\oplus a}$ is free. This shows that the nonisomorphic indecomposable $R$-modules in $\mathsf{\Omega CM}(R)$ are $R$ and $S(\cong \mathfrak{m})$, which especially says that $R$ is of finite $\mathsf{\Omega CM}$ type. On the other hand, $R$ is not of finite $\mathsf{CM}$ type when $n\geq 4$ (see \cite[Theorem 4.10]{LW}).
\end{ex}

\begin{ac}
The author is grateful to his supervisor Ryo Takahashi for giving him helpful advice throughout the paper. The author also thanks Osamu Iyama for coming up with the ring $R$ in Example \ref{20}.
\end{ac}

%%%%%%%%%%%%%%%%%%%%%%%%%%%%%%%%%%%%%%%%%%%%%%%%%%%%%%%%

%%%%%%%%%%%%%%%%%%%%%%%%%%%%%%%%%%%%%%%%%%%%%%%%%%%%%%%%
\end{document}